\documentclass[11pt,times,twoside]{article}
\usepackage[body={14.5cm,22cm}, left=2.5cm, right=2.5cm, top=3cm]{geometry}
\usepackage{graphicx,amsmath,amsthm,latexsym,amssymb,amstext}
\usepackage[pdfstartview=FitH,
bookmarksnumbered=true,bookmarksopen=true,%
colorlinks=true,pdfborder=001,citecolor=blue,%
linkcolor=blue,urlcolor=blue]{hyperref}

\newtheorem{theorem}{Theorem}[section]

\newtheorem{proposition}[theorem]{Proposition}
\newtheorem{lemma}[theorem]{Lemma}

\newtheorem{remark}[theorem]{Remark}

\numberwithin{equation}{section}

\begin{document}

\title{\Large \bf Anticipative backward stochastic differential equations driven by fractional Brownian motion }

\author
{\textbf{Jiaqiang Wen}$^{1}$, \textbf{Yufeng Shi}$^{2,1,}$\thanks{Corresponding author. Emails: jqwen59@gmail.com (J. Wen), yfshi@sdu.edu.cn (Y. Shi)} \\
\normalsize{$^{1}$Institute for Financial Studies and School of Mathematics,}\\
\normalsize{Shandong University, Jinan 250100, China.}\\
\normalsize{$^{2}$School of Statistics, Shandong University of Finance and Economics, Jinan 250014, China}\\
}

\date{}

\renewcommand{\thefootnote}{\fnsymbol{footnote}}

\footnotetext[0]{The first and second author are supported  by National Natural Science Foundation of China (Grant Nos. 11371226 and 11231005),
Foundation for Innovative Research Groups of National Natural Science Foundation of China (Grant  No. 11221061), the 111 Project (Grant No. B12023).}

\maketitle

\begin{abstract}
We study the anticipative backward stochastic differential equations (BSDEs, for short) driven by
 fractional Brownian motion with Hurst parameter H greater than 1/2.
The stochastic integral used throughout the paper is the divergence operator type integral.
We obtain the existence and uniqueness theorem to these equations under the Lipschitz condition.
A comparison theorem for this type of anticipative BSDE is also established.
\end{abstract}



\section{Introduction}

Fractional Brownian motion (fBm, for short) with Hurst parameter $H \in (0, 1)$ is a zero mean Gaussian process $B^{H}=\{ B^{H}_{t},t\geq 0 \}$
whose covariance is given by
\begin{equation*}
  \mathbb{E}(B^{H}_{t}B^{H}_{s}) = \frac{1}{2} (t^{2H} + s^{2H} - |t-s|^{2H}).
\end{equation*}
For $H=\frac{1}{2}$, the process $B^{H}$ is a classical Brownian motion.
In the case $H > \frac{1}{2}$, the process $B^{H}$ exhibits long range dependence.
These properties make this process a useful driving noise in models arising in  finance, physics, telecommunication networks and other fields.
However, since $B^{H}$ with $H> \frac{1}{2}$ is not a semimartingale, we cannot use the classical theory of stochastic calculus to define
the fractional stochastic integral.
Essentially, two different types of integrals with respect to fBm have been defined and developed.
The first one is the pathwise Riemann-Stieltjes integral which exists if the integrand has a continuous paths of order
 $\alpha >1 - H$ (see Young \cite{Young}).
This type of integral has the properties of Stratonovich integral, which can lead to difficulties in applications.
The second one, introduced by Decreusefond-\"{U}st\"{u}nel \cite{Decreusefond}, is the divergence operator (Skorohod
integral), defined as the adjoint of the derivative operator in the framework of Malliavin calculus.
Since this stochastic integral satisfies the zero mean property and it can be expressed as the limit of Riemann sums by using Wick products,
it was later developed by many authors. We refer to the works of Biagini-Hu-{\O}ksendal-Zhang \cite{Biagini} and Nualart \cite{Nualart}.

Backward stochastic differential equations (BSDEs in short) driven by Brownian motion were introduced by Bismut \cite{Bismut}
for the linear case and by Pardoux-Peng \cite{Peng} in the general case.
Since then, these pioneer works are extensively used in many fields like mathematical finance \cite{Peng2}, stochastic optimal control and stochastic games \cite{Hama}.
Recently, Peng-Yang \cite{Yang} introduced a new type of BSDEs, called anticipative BSDEs, which can be regarded as a new duality type of stochastic differential delay equations.
BSDEs driven by fBm were firstly studied by Hu \cite{Hu3} and Hu-Peng \cite{Hu},
where they obtained the existence and uniqueness of the solution.
Then Hu-Ocone-Song \cite{Hu1} established a comparison result for fractional BSDEs.
Some other recent developments of fractional BSDEs  can be found in
Bender \cite{Bender}, Borkowska \cite{Borkowska} and Jing \cite{Jing}, etc.

Motivated by the above works, the purpose of this paper is to study the anticipative BSDEs driven by fBm with Hurst parameter $H > \frac{1}{2}$.
Under the Lipschitz condition, we prove this type of equation admits a unique solution.
As a fundamental tool, the comparison theorem plays an important role in the theory and applications of BSDEs.
We also establish a comparison theorem for this class of anticipative BSDEs.

This paper is organized as follows. In Section 2, we provide some basic results on fractional Brownian motions.
Section 3 contains the definition of anticipative BSDEs with respect to the fBm.
The existence and uniqueness result is proved here.
We give a comparison theorem for the solutions of anticipative BSDEs in Section 4.

\section{Fractional calculus}

Let $(\Omega,\mathcal{F}, P, \mathcal{F}_{t}, t\geq 0)$ be a complete stochastic basis such that $\mathcal{F}_{0}$ contains all $P$-null elements
 of $\mathcal{F}$ and suppose that the filtration is generated by a fractional Brownian motion $B^{H}=\{ B^{H}_{t},t\geq 0 \}$.
We assume $H > \frac{1}{2}$ throughout this paper.
Denote $\phi(x) = H(2H - 1)|x|^{2H-2}, \ x \in \mathbb{R}$.
Let $\xi$ and $\eta$ be two continuous functions on $[0,T]$.
We define
\begin{equation*}
  \langle \xi,\eta \rangle_{t} = \int_0^t \int_0^t \phi(u-v) \xi_{u} \eta_{v} dudv,
\end{equation*}
and $ \| \xi \|_{t}^{2} =  \langle \xi,\xi  \rangle_{t}$.
Note that, for any $t\in [0,T], \ \langle \xi,\eta \rangle_{t}$ is a Hilbert scalar product.
Let $\mathcal{H}$ be the completion of the continuous functions under this Hilbert norm.
The elements of $\mathcal{H}$ may be distributions.

We denote by $\mathcal{P}_{T}$ the set of all polynomials of fractional Brownian motion in $[0,T]$, i.e., it contains all elements of the form
\begin{equation*}
  F(\omega) = f \left(\int_0^T \xi_{1}(t) dB_{t}^{H},...,\int_0^T \xi_{n}(t) dB_{t}^{H} \right),
\end{equation*}
where $f$ is a polynomial function of $n$ variables.
The Malliavin derivative operator $D_{s}^{H}$ of an element $F\in \mathcal{P}_{T}$ is defined as follows:
\begin{equation*}
  D_{s}^{H}F = \sum\limits_{i=1}^{n} \frac{\partial f}{\partial x_{i}}
               \left(\int_0^T \xi_{1}(t) dB_{t}^{H},...,\int_0^T \xi_{n}(t) dB_{t}^{H} \right)\xi_{i}(s), \ \ s\in [0,T].
\end{equation*}
Since the divergence operator
  $D^{H}:L^{2}(\Omega,\mathcal{F}, P) \rightarrow (\Omega,\mathcal{F}, \mathcal{H})$ is closable,
we can consider the space $\mathbb{D}^{1,2}$ be the completion of $\mathcal{P}_{T}$ with the norm
\begin{equation*}
  \| F \|^{2}_{1,2} = \mathbb{E}|F|^{2} + \mathbb{E}\|D^{H}_{s} F\|^{2}_{T}.
\end{equation*}
We also introduce another derivative
\begin{equation*}
  \mathbb{D}_{t}^{H}F = \int_0^T \phi(t-s) D_{s}^{H}F ds.
\end{equation*}
Denote by $\mathbb{L}^{1,2}_{H}$ the space of all stochastic processes $F : (\Omega,\mathcal{F}, P) \rightarrow \mathcal{H}$ such that
\begin{equation*}
  \mathbb{E} \left( \| F \|_{T}^{2} + \int_0^T \int_0^T |\mathbb{D}_{s}^{H}F_{t}|^{2} dsdt \right) < \infty.
\end{equation*}

The following results are well know now (see \cite{Hu2,Hu3,Hu4}).

\begin{proposition}\label{2}
Let $F \in \mathbb{L}^{1,2}_{H}$, then the It\^{o}-Skorohod type stochastic integral
$\int_0^T F_{s} dB_{s}^{H}$ exists in $L^{2}(\Omega,\mathcal{F}, P)$. Moreover, we have
\begin{equation*}
 \mathbb{E} \left( \int_0^T F_{s} dB_{s}^{H} \right) = 0,
\end{equation*}
and
\begin{equation*}
   \mathbb{E} \left( \int_0^T F_{s} dB_{s}^{H} \right)^{2}
 = \mathbb{E} \left( \| F \|_{T}^{2} + \int_0^T \int_0^T \mathbb{D}_{s}^{H}F_{t} \mathbb{D}_{t}^{H}F_{s} dsdt \right).
\end{equation*}
\end{proposition}

\begin{proposition}
Let $f, g: [0, T] \rightarrow \mathbb{R}$ be deterministic continuous functions. If
\begin{equation*}
  X_{t} = X_{0} + \int_0^t g_{s} ds + \int_0^t f_{s} dB_{s}^{H}, \ \ t\in [0,T],
\end{equation*}
where $X_{0}$ is a constant and $F \in C^{1,2}([0, T ] \times \mathbb{R})$, then for any $t\in [0,T]$,
\begin{equation*}
 F(t,X_{t})= F(0,X_{0})+ \int_0^t \frac{\partial F}{\partial s}(s,X_{s}) ds + \int_0^t \frac{\partial F}{\partial x}(s,X_{s}) dX_{s}
             + \frac{1}{2}\int_0^t \frac{\partial^{2} F}{\partial x^{2}}(s,X_{s}) \bigg[\frac{d}{ds} \| f \|_{s}^{2}\bigg] ds.
\end{equation*}
\end{proposition}

\begin{proposition}\label{224}
Let $f_{i}(s), g_{i}(s)$ be in $\mathbb{D}^{1,2}$ and $\mathbb{E} \int_0^T (|f_{i}(s)|^{2} + |g_{i}(s)|^{2}) ds < \infty$, where $i = 1, 2$.
Assume that $D^{H}_{t}f_{1}(s)$ and $D^{H}_{t}f_{2}(s)$ are continuously differentiable with respect to
 $(s,t)\in [0,T]^{2}$ for almost all $\omega \in \Omega$.
Suppose also that $\mathbb{E} \int_0^T \int_0^T |\mathbb{D}_{t}^{H}f_{i}(s)|^{2} dsdt < \infty$.
Denote
\begin{equation*}
  X_{i}(t) = \int_0^t g_{i}(s) ds + \int_0^t f_{i}(s) dB_{s}^{H}, \ \ t\in [0,T].
\end{equation*}
Then
\begin{equation*}
\begin{split}
   X_{1}(t)X_{2}(t) =& \int_0^t X_{1}(s)g_{2}(s) ds + \int_0^t X_{1}(s)f_{2}(s) dB_{s}^{H}
                      +\int_0^t X_{2}(s)g_{1}(s) ds \\
                     &+ \int_0^t X_{2}(s)f_{1}(s) dB_{s}^{H}
                      +\int_0^t \mathbb{D}_{s}^{H}X_{1}(s)g_{2}(s) ds + \int_0^t \mathbb{D}_{s}^{H}X_{2}(s)g_{1}(s) ds.
\end{split}
\end{equation*}
\end{proposition}

\section{Anticipative BSDEs}

In this section, we study the anticipative BSDEs driven by fBm.
The existence and uniqueness theorem  is proved here.
In the following, let
\begin{equation*}
\eta_{t} = \eta_{0} + \int_0^t b_{s} ds + \int_0^t \sigma_{s} dB_{s}^{H},
\end{equation*}
where $\eta_{0}$ is a constant, $b$ and $\sigma$ are two deterministic differentiable functions,
such that $\sigma_{t} \neq 0$ (then either $\sigma_{t} < 0$ or $\sigma_{t} > 0), \ t\in [0,T]$.
Note that, since
\begin{equation*}
  \| \sigma \|_{t}^{2} = H(2H-1) \int_0^t \int_0^t |u-v|^{2H-2} \sigma_{u} \sigma_{v} dudv,
\end{equation*}
we have $\frac{d}{dt}(\| \sigma \|_{t}^{2})= 2 \hat{\sigma}_{t} \sigma_{t}>0$ for $t\in (0,T]$, where
$\hat{\sigma}_{t} = \int_0^t \phi(t-v) \sigma_{v} dv$.

Motivated by Hu-Peng \cite{Hu}, we now introduce the anticipative BSDEs driven by fBm as follows:
\begin{equation}\label{31}
  \begin{cases}
    -dY_{t}=f(t,\eta_{t},Y_{t},Z_{t},Y_{t+\delta(t)},Z_{t+\zeta(t)}) dt - Z_{t} dB_{t}^{H}, \ \ \ t\in [0,T]; \\
      Y_{t}=g(\eta_{t}), \ \  Z_{t}=h(\eta_{t}), \ \ \ \ \ \ \ \ \ \ \ \ \ \ \ \ \ \ \ \ \ \ \ \ \ \ \ \ \ \  t\in[T,T+K],
  \end{cases}
\end{equation}
where $\delta(\cdot)$ and $\zeta(\cdot)$ are two deterministic $\mathbb{R}^{+}$-valued continuous functions defined on $[0,T]$ such that:
\begin{itemize}
  \item [(i)] There exists a constant $K\geq 0$ such that, for all $t\in[0,T]$,
    \begin{equation*}
      t + \delta(t) \leq T+K; \ \ \ \ t + \zeta(t) \leq T+K.
    \end{equation*}
  \item [(ii)] There exists a constant $L \geq 0$ such that, for all $t \in [0,T]$ and for all nonnegative and integrable $m(\cdot)$,
   \begin{equation*}
     \int_t^T m(s + \delta(s)) ds \leq L \int_t^{T+K} m(s) ds; \ \ \
     \int_t^T m(s + \zeta(s)) ds \leq L \int_t^{T+K} m(s) ds.
   \end{equation*}
\end{itemize}
We introduce the following sets:
\begin{itemize}
  \item [$\bullet$] $C^{1,3}_{pol}([0,T] \times \mathbb{R})$
        is the space of all $C^{1,3}$-functions over $[0, T] \times \mathbb{R}$, which together with their derivatives are of polynomial growth;
  \item [$\bullet$] $L^{2}(\mathcal{F}_{r};\mathbb{R}) =\{\xi:\Omega\rightarrow \mathbb{R} \mid \xi$
        is $\mathcal{F}_{r}$-measurable, $\mathbb{E}[|\xi|^{2}]< \infty\}$;
  \item [$\bullet$] $L_{\mathcal{F}}^2(0,T;\mathbb{R})
  =\{X:\Omega\times [0,T]\rightarrow \mathbb{R} \mid X(\cdot)$ is $\mathcal{F}$-adapted,
  $\mathbb{E}\int_0^T |X(t)|^{2} dt< \infty\}$;
\end{itemize}
%
%$\bullet$ $C^{1,3}_{pol}([0,T] \times \mathbb{R})$
% is the space of all $C^{1,3}$-functions over $[0, T] \times \mathbb{R}$, which together with
%their derivatives are of polynomial growth;\\
% $\bullet$ $L^{2}(\mathcal{F}_{r};\mathbb{R})
%  =\{\xi:\Omega\rightarrow \mathbb{R} \mid \xi$ is $\mathcal{F}_{r}$-measurable, $\mathbb{E}[|\xi|^{2}]< \infty\}$;\\
%$\bullet$ $L_{\mathcal{F}}^2(0,T;\mathbb{R})
%  =\{X:\Omega\times [0,T]\rightarrow \mathbb{R} \mid X(\cdot)$ is $\mathcal{F}$-adapted,
%  $\mathbb{E}\int_0^T |X(t)|^{2} dt< \infty\}$;\\
 $\mathcal{V}_{[0,T]} := \bigg \{ Y=\varphi\big(\cdot,\eta(\cdot)\big) | \varphi \in C_{pol}^{1,3}([0,T]\times \mathbb{R})
  \  with \ \frac{\partial \varphi}{\partial t}  \in C_{pol}^{0,1}([0,T]\times \mathbb{R}), t\in[0,T] \bigg \},$\\
and by $\widetilde{\mathcal{V}}_{[0,T+K]}$ and $\widetilde{\mathcal{V}}_{[0,T+K]}^{H}$  denote the completion of $\mathcal{V}_{[0,T+K]}$ under the following norm respectively,
\begin{equation*}
  \| Y \| := \bigg(\int_0^{T+K}  e^{\beta t} \mathbb{E} |Y(t)|^{2} dt\bigg)^{\frac{1}{2}}, \ \ \ \
  \| Z \| := \bigg(\int_0^{T+K} t^{2H-1} e^{\beta t} \mathbb{E} |Z(t)|^{2} dt\bigg)^{\frac{1}{2}}.
\end{equation*}
It's easy to know  $\widetilde{\mathcal{V}}_{[0,T+K]}^{H} \subseteq \widetilde{\mathcal{V}}_{[0,T+K]} \subseteq L^{2}_{\mathcal{F}}(0,T+K;\mathbb{R})$.

The setting of our problem is as follows: to find a pair of processes
 $(Y_{\cdot},Z_{\cdot}) \in \widetilde{\mathcal{V}}_{[0,T+K]} \times \widetilde{\mathcal{V}}^{H}_{[0,T+K]}$ satisfying the anticipative BSDE (\ref{31}).

\begin{itemize}
\item[(H1)] $g$ and $h$ are given elements in $ C^{2}_{pol}(\mathbb{R})$ such that
\begin{equation*}
\mathbb{E}\int_T^{T+K} e^{\beta t}|g(\eta_{t})|^{2} dt<0; \ \
  \mathbb{E}\int_T^{T+K} e^{\beta t}t^{2H-1}|h(\eta_{t})|^{2} dt<0.
\end{equation*}
\end{itemize}
The driver $f(t,x,y,z,\theta,\zeta): [0,T]\times\mathbb{R}^{3}\times L^{2}(\mathcal{F}_{r'},\mathbb{R}) \times L^{2}(\mathcal{F}_{r},\mathbb{R})
                 \longrightarrow L^{2}(\mathcal{F}_{t},\mathbb{R})$ is a $C_{pol}^{0,1}$-continuous function, where $r', r \in[t,T+K]$
and $f$ satisfies the following condition:
\begin{itemize}
\item[(H2)]
 There exists a constant $C\geq 0$, such that for all
$t\in [0,T], x,y,y',z,z' \in \mathbb{R},$ $\theta_{\cdot},\theta'_{\cdot},\zeta_{\cdot},\zeta'_{\cdot} \in L_{\mathcal{F}}^2(t,T+K;\mathbb{R})$, we have
\begin{equation*}
 |f(t,x,y,z,\theta_{r'},\zeta_{r}) - f(t,x,y',z',\theta'_{r'},\zeta'_{r})|
       \leq C\big(|y-y'| +|z-z'| + \mathbb{E}^{\mathcal{F}_{t}}[|\xi_{r'}-\xi'_{r'}|+|\zeta_{r}-\zeta'_{r}|]\big).
\end{equation*}
%\begin{equation*}
%     and \  \mathbb{E}\int_0^T e^{\beta s}|f_{0}(s,\eta_{s})|^{2} dt<0, \ where \ f_{0}(s,x)=f(s,x,0,0,0,0). \ \ \ \ \ \ \ \ \ \ \ \  \ \ \ \ \ \ \ \ \ \ \ \ \ \ \ \ \ \ \ \
%\end{equation*}
\end{itemize}
\begin{lemma}\label{35}
Suppose $g$ is a given differentiable function with polynomial growth, and $f(t,x)$ is a $C_{pol}^{0,1}$-continuous function.
Then BSDE
\begin{equation}\label{33}
    Y_{t}=g(\eta_{T}) + \int_t^T f(s,\eta_{s}) ds - \int_t^T Z_{s} dB_{s}^{H},
\end{equation}
admits a unique solution
 $(Y_{\cdot},Z_{\cdot}) \in \widetilde{\mathcal{V}}_{[0,T]} \times \widetilde{\mathcal{V}}^{H}_{[0,T]}$,
and the following estimate holds,
\begin{align}\label{34}
    & \mathbb{E}\left(e^{\beta t}|Y_{t}|^{2} + \frac{\beta}{2}\int_t^T e^{\beta s}|Y_{s}|^{2} ds
    + \frac{2}{M}\int_t^Te^{\beta s}s^{2H-1}|Z_{s}|^{2} ds\right) \nonumber\\
\leq& \mathbb{E}\left(e^{\beta T}|g(\eta_{T})|^{2} + \frac{2}{\beta}\int_t^T e^{\beta s}|f(s,\eta_{s})|^{2} ds \right).
\end{align}
where $M>0$ is a suitable constant and $\beta> 0$.
\end{lemma}

\begin{proof}
From Theorem 4 of Borkowska \cite{Borkowska}, we know that Eq. (\ref{33}) has a unique solution
$(Y_{\cdot},Z_{\cdot}) \in \widetilde{\mathcal{V}}_{[0,T]} \times \widetilde{\mathcal{V}}_{[0,T]}^{H}$.
By the It\^{o} formula, we obtain
\begin{equation*}
  \begin{split}
   e^{\beta t}Y_{t}^{2}=& e^{\beta T}g(\eta_{T})^{2} - \beta \int_t^T e^{\beta s}Y_{s}^{2} ds + 2\int_t^T e^{\beta s}Y_{s}f(s,\eta_{s}) ds\\
                        & -2\int_t^T e^{\beta s}Y_{s}Z_{s} dB_{s}^{H} -2\int_t^T e^{\beta s}\mathbb{D}_{s}^{H}Y_{s}Z_{s} ds.
  \end{split}
\end{equation*}
It is known (see example Hu-Peng \cite{Hu}) that $\mathbb{D}_{s}^{H} Y_{s} = \frac{\hat{\sigma}_{s}}{\sigma_{s}}Z_{s}$.
Moreover by Remark 6 in Maticiuc-Nie \cite{Maticiuc}, there exists $M>0$ such that for all $t\in [0,T]$,
$\frac{t^{2H-1}}{M}\leq \frac{\hat{\sigma}_{t}}{\sigma_{t}}\leq M t^{2H-1}$.
Thus we have
\begin{align}\label{36}
    & \mathbb{E}\left(e^{\beta t}Y_{t}^{2} + \beta \int_t^T e^{\beta s}Y_{s}^{2} ds+ \frac{2}{M}\int_t^T e^{\beta s}s^{2H-1}Z_{s}^{2} ds\right) \nonumber\\
\leq& \mathbb{E}\left(e^{\beta T}g(\eta_{T})^{2} + 2\int_t^T e^{\beta s}Y_{s}f(s,\eta_{s}) ds\right) \\
\leq& \mathbb{E}\left(e^{\beta T}g(\eta_{T})^{2} + \frac{\beta}{2}\int_t^T e^{\beta s}Y_{s}^{2} ds
    +  \frac{2}{\beta}\int_t^T e^{\beta s}|f(s,\eta_{s})|^{2} ds \right). \nonumber
\end{align}
Then we obtain the estimate (\ref{34}).
\end{proof}

The following theorem is the main result of this section: an existence and uniqueness theorem for anticipative BSDEs with respect to the fBm.
\begin{theorem}\label{30}
Let (H1) and (H2) hold, and $\delta, \zeta$ satisfy (i) and (ii).
Then the anticipative BSDE (\ref{31}) admits a unique solution
 $(Y_{\cdot},Z_{\cdot})\in\widetilde{\mathcal{V}}_{[0,T+K]} \times \widetilde{\mathcal{V}}_{[0,T+K]}^{H}$.
 Moreover, for all $t\in[0,T]$,
\begin{equation}\label{20}
      \mathbb{E}\left( e^{\beta t}|Y_{t}|^{2} + \int_t^T e^{\beta s}s^{2H-1}|Z_{s}|^{2} ds\right)
 \leq R\Theta(t,T,K),
\end{equation}
where $R$ is a positive constant which may be different from line to line, and
\begin{equation*}
  \Theta(t,T,K)= \mathbb{E}\bigg(e^{\beta T}|g(\eta_{T})|^{2}  + \int_t^T e^{\beta s}|f_{0}(s,\eta_{s})|^{2} ds
                + \int_T^{T+K} e^{\beta s}\big(|g(\eta_{s})|^{2}+s^{2H-1}|h(\eta_{s})|^{2}\big) ds \bigg).
\end{equation*}
\end{theorem}

\begin{proof}
The method used here is similar to that in the proof of Proposition 19 in Maticiuc-Nie \cite{Maticiuc}.
For any given $(y_{t},z_{t}) \in \widetilde{\mathcal{V}}_{[0,T+K]} \times \widetilde{\mathcal{V}}^{H}_{[0,T+K]}$, we consider the following BSDE:
\begin{equation}\label{32}
  \begin{cases}
    -dY_{t}=f(t,\eta_{t},y_{t},z_{t},y_{t+\delta(t)},z_{t+\zeta(t)}) dt - Z_{t} dB_{t}^{H}, \ \ \ t\in [0,T]; \\
      Y_{t}=g(\eta_{t}), \ \  Z_{t}=h(\eta_{t}), \ \ \ \ \ \ \ \ \ \ \ \ \ \ \ \ \ \ \ \ \ \ \ \ \ \ \ \ t\in[T,T+K].
  \end{cases}
\end{equation}
From Lemma \ref{35}, we know (\ref{32}) has a unique solution
$(Y_{\cdot},Z_{\cdot}) \in \widetilde{\mathcal{V}}_{T+K} \times \widetilde{\mathcal{V}}^{H}_{T+K}$.
Define a mapping
$I:\widetilde{\mathcal{V}}_{[0,T+K]} \times \widetilde{\mathcal{V}}^{H}_{[0,T+K]}\longrightarrow
\widetilde{\mathcal{V}}_{[0,T+K]} \times \widetilde{\mathcal{V}}^{H}_{[0,T+K]}$
such that $I[(y_{\cdot},z_{\cdot})]=(Y_{\cdot},Z_{\cdot})$.
Note that since $Y_{t}=g(\eta_{t})$ and $Z_{t}=h(\eta_{t})$ are given when $t\in[T,T+K]$,
we essentially need to prove (\ref{31}) has a unique solution on $[0,T]$.
Let $n\in \mathbb{N}$ and $t_{i}=\frac{i-1}{n}T, i=1,...,n+1$.
First we will solve (\ref{31}) on $[t_{n},T]$.
In order to do this, we show $I$ is a contraction on $\widetilde{\mathcal{V}}_{[t_{n},T+K]} \times \widetilde{\mathcal{V}}^{H}_{[t_{n},T+K]}$.

For two arbitrary elements $(y_{\cdot},z_{\cdot})$ and
 $(y'_{\cdot},z'_{\cdot})\in \widetilde{\mathcal{V}}_{[t_{n},T+K]} \times \widetilde{\mathcal{V}}^{H}_{[t_{n},T+K]}$,
set $(Y_{\cdot},Z_{\cdot})=I[(y_{\cdot},z_{\cdot})]$ and $(Y'_{\cdot},Z'_{\cdot})=I[(y'_{\cdot},z'_{\cdot})]$.
We denote their differences by
\begin{equation*}
    (\hat{y}_{\cdot},\hat{z}_{\cdot})=((y_{\cdot}-y'_{\cdot}),(z_{\cdot}-z'_{\cdot})), \ \ \
  (\hat{Y}_{\cdot},\hat{Z}_{\cdot})=((Y_{\cdot}-Y'_{\cdot}),(Z_{\cdot}-Z'_{\cdot})).
\end{equation*}
From the It\^{o} formula, for $t\in [t_{n},T]$, similarly as (\ref{36}),
\begin{equation*}
\begin{split}
    &  \mathbb{E}\left(e^{\beta t}\hat{Y}_{t}^{2} + \beta \int_t^T e^{\beta s}\hat{Y}_{s}^{2} ds
     + \frac{2}{M}\int_t^T e^{\beta s}s^{2H-1}\hat{Z}_{s}^{2} ds\right) \\
\leq&  2 \mathbb{E}\int_t^T e^{\beta s}
      \hat{Y}_{s}\big[f(s,\eta_{s},y_{s},z_{s},y_{s+\delta(s)},z_{s+\zeta(s)}) - f(s,\eta_{s},y'_{s},z'_{s},y'_{s+\delta(s)},z'_{s+\zeta(s)})\big] ds.
\end{split}
\end{equation*}
Choose $\beta>1$ and $M>2$. Then from the assumption (H2) and the Schwarz inequality one has
\begin{align}
    &  \mathbb{E}\left(e^{\beta t}|\hat{Y}_{t}|^{2} + \int_t^T e^{\beta s}|\hat{Y}_{s}|^{2} ds
       + \frac{2}{M}\int_t^T e^{\beta s}s^{2H-1}|\hat{Z}_{s}|^{2} ds\right) \nonumber\\
\leq&  2C\int_t^Te^{\beta s}\mathbb{E}\big[|\hat{Y}_{s}| ( |\hat{y}_{s}| + |\hat{z}_{s}|)\big] ds
       +2C\int_t^T e^{\beta s} \mathbb{E}\big[|\hat{Y}_{s}| ( |\hat{y}_{s+\delta(s)}| + |\hat{z}_{s+\zeta(s)}|)\big] ds \label{37}\\
\leq&  2C\int_t^T (e^{\beta s} \mathbb{E}|\hat{Y}_{s}|^{2})^{\frac{1}{2}} \bigg(\big[e^{\beta s} \mathbb{E}(|\hat{y}_{s}| + |\hat{z}_{s}|)^{2}\big]^{\frac{1}{2}}
       + \big[e^{\beta s} \mathbb{E}(|\hat{y}_{s+\delta(s)}| + |\hat{z}_{s+\zeta(s)}|)^{2}\big]^{\frac{1}{2}} \bigg) ds \label{38}.
\end{align}
Denote $x(t)=(e^{\beta t} \mathbb{E}|\hat{Y}_{t}|^{2})^{\frac{1}{2}}$. Then from (\ref{38}),
\begin{equation*}
  x(t)^{2}\leq  2C\int_t^T x(s) \bigg(\big[e^{\beta s} \mathbb{E}(|\hat{y}_{s}| + |\hat{z}_{s}|)^{2}\big]^{\frac{1}{2}}
       + \big[e^{\beta s} \mathbb{E}(|\hat{y}_{s+\delta(s)}| + |\hat{z}_{s+\zeta(s)}|)^{2}\big]^{\frac{1}{2}} \bigg) ds.
\end{equation*}
Applying Lemma 20 in Maticiuc-Nie \cite{Maticiuc} to the above inequality, it follows that
\begin{align}
  x(t)\leq& C\int_t^T \bigg(\big[e^{\beta s} \mathbb{E}(|\hat{y}_{s}| + |\hat{z}_{s}|)^{2}\big]^{\frac{1}{2}}
       + \big[e^{\beta s} \mathbb{E}(|\hat{y}_{s+\delta(s)}| + |\hat{z}_{s+\zeta(s)}|)^{2}\big]^{\frac{1}{2}} \bigg) ds \nonumber \\
      \leq& \sqrt{2}C \int_t^T \big[e^{\beta s} \mathbb{E}(|\hat{y}_{s}|^{2} + |\hat{z}_{s}|^{2})\big]^{\frac{1}{2}} ds
       + \sqrt{2}C \int_t^T \big[e^{\beta s} \mathbb{E}(|\hat{y}_{s+\delta(s)}|^{2} + |\hat{z}_{s+\zeta(s)}|^{2})\big]^{\frac{1}{2}} ds. \nonumber
\end{align}
Therefore for $t\in [t_{n},T]$,
\begin{align*}
  x(t)^{2}\leq 4C^{2} \bigg(\int_t^T \big[e^{\beta s} \mathbb{E}(|\hat{y}_{s}|^{2} + |\hat{z}_{s}|^{2})\big]^{\frac{1}{2}} ds\bigg)^{2}
       + 4C^{2} \bigg(\int_t^T \big[e^{\beta s} \mathbb{E}(|\hat{y}_{s+\delta(s)}|^{2} + |\hat{z}_{s+\zeta(s)}|^{2})\big]^{\frac{1}{2}} ds \bigg)^{2}.
\end{align*}
Now we compute
\begin{align}
  \int_{t_{n}}^T x(s)^{2} ds
  \leq& 4C^{2}(T-t_{n}) \bigg(\int_{t_{n}}^T \big[e^{\beta s} \mathbb{E}(|\hat{y}_{s}|^{2} + |\hat{z}_{s}|^{2})\big]^{\frac{1}{2}} ds\bigg)^{2} \nonumber\\
      &+4C^{2}(T-t_{n})\bigg(\int_{t_{n}}^T \big[e^{\beta s} \mathbb{E}(|\hat{y}_{s+\delta(s)}|^{2}
       +|\hat{z}_{s+\zeta(s)}|^{2})\big]^{\frac{1}{2}}ds\bigg)^{2}. \nonumber\\
    =:& A_{1} + A_{2}.\label{39}
\end{align}
For  the term $A_{2}$ in (\ref{39}),
\begin{align}
      &\bigg(\int_{t_{n}}^T \big[e^{\beta s} \mathbb{E}(|\hat{y}_{s+\delta(s)}|^{2}+|\hat{z}_{s+\zeta(s)}|^{2})\big]^{\frac{1}{2}}ds\bigg)^{2} \nonumber \\
 \leq& \bigg(\int_{t_{n}}^T \big[e^{\beta s} \mathbb{E}|\hat{y}_{s+\delta(s)}|^{2} \big]^{\frac{1}{2}} ds
      +\int_{t_{n}}^T \big[e^{\beta s} \mathbb{E}|\hat{z}_{s+\zeta(s)}|^{2}\big]^{\frac{1}{2}}ds\bigg)^{2} \nonumber\\
 \leq& 2\bigg(\int_{t_{n}}^T \big[e^{\beta s} \mathbb{E}|\hat{y}_{s+\delta(s)}|^{2} \big]^{\frac{1}{2}} ds \bigg)^{2}
      +2\bigg(\int_{t_{n}}^T \big[\frac{1}{s^{2H-1}} \cdot e^{\beta s} s^{2H-1}\mathbb{E}|\hat{z}_{s+\zeta(s)}|^{2}\big]^{\frac{1}{2}}ds\bigg)^{2} \nonumber\\
 \leq& 2(T-t_{n})\int_{t_{n}}^Te^{\beta s} \mathbb{E}|\hat{y}_{s+\delta(s)}|^{2} ds
      +\frac{2(T^{2-2H}-t_{n}^{2-2H})}{2-2H} \int_{t_{n}}^Te^{\beta s} s^{2H-1}\mathbb{E}|\hat{z}_{s+\zeta(s)}|^{2} ds \nonumber\\
 \leq& \big[2(T-t_{n})+ \frac{T^{2-2H}-t_{n}^{2-2H}}{1-H}\big]
      \mathbb{E} \int_{t_{n}}^T \big[e^{\beta (s+\delta(s))}|\hat{y}_{s+\delta(s)}|^{2}
       + e^{\beta (s+\zeta(s))}(s+\zeta(s))^{2H-1}|\hat{z}_{s+\zeta(s)}|^{2}\big] ds \nonumber\\
 \leq& \big[2(T-t_{n})+ \frac{T^{2-2H}-t_{n}^{2-2H}}{1-H}\big]L\cdot
      \mathbb{E} \int_{t_{n}}^{T+K} e^{\beta s}\big(|\hat{y}_{s}|^{2} + s^{2H-1}|\hat{z}_{s}|^{2}\big) ds.  \label{302}
\end{align}
Similarly, for  the term $A_{1}$ in (\ref{39}),
\begin{align}
      &\bigg(\int_{t_{n}}^T \big[e^{\beta s}\mathbb{E} (|\hat{y}_{s}|^{2} + |\hat{z}_{s}|^{2})\big]^{\frac{1}{2}} ds\bigg)^{2}\nonumber\\
 \leq& \big[2(T-t_{n})+ \frac{T^{2-2H}-t_{n}^{2-2H}}{1-H}\big]
      \mathbb{E} \int_{t_{n}}^{T} e^{\beta s}\big(|\hat{y}_{s}|^{2} + s^{2H-1}|\hat{z}_{s}|^{2}\big) ds\nonumber\\
 \leq& \big[2(T-t_{n})+ \frac{T^{2-2H}-t_{n}^{2-2H}}{1-H}\big]
      \mathbb{E} \int_{t_{n}}^{T+K} e^{\beta s}\big(|\hat{y}_{s}|^{2} + s^{2H-1}|\hat{z}_{s}|^{2}\big) ds.   \label{303}
\end{align}
Combining (\ref{39}-\ref{303}), it follows that
\begin{align}\label{304}
  \int_{t_{n}}^T x(s)^{2} ds
  \leq(T-t_{n})G\cdot \mathbb{E} \int_{t_{n}}^{T+K} e^{\beta s}\big(|\hat{y}_{s}|^{2} + s^{2H-1}|\hat{z}_{s}|^{2}\big) ds,
\end{align}
where $G=4C^{2} (L+1)\big[2(T-t_{n})+ \frac{T^{2-2H}-t_{n}^{2-2H}}{1-H}\big]$. And similarly,
\begin{align}\label{305}
  \int_{t_{n}}^T \frac{1}{s^{2H-1}}x(s)^{2} ds
  \leq G \frac{T^{2-2H}-t_{n}^{2-2H}}{2-2H} \mathbb{E} \int_{t_{n}}^{T+K} e^{\beta s}\big(|\hat{y}_{s}|^{2} + s^{2H-1}|\hat{z}_{s}|^{2}\big) ds.
\end{align}
Now from (\ref{37}),
\begin{align}
    &  \mathbb{E}\left( \int_{t_{n}}^T e^{\beta s}|\hat{Y}_{s}|^{2} ds
       + \frac{2}{M}\int_{t_{n}}^T e^{\beta s}s^{2H-1}|\hat{Z}_{s}|^{2} ds\right) \nonumber\\
\leq&  2C \mathbb{E}\int_{t_{n}}^T e^{\beta s} \bigg(\frac{1}{v}(1+\frac{1}{s^{2H-1}})|\hat{Y}_{s}|^{2}
        + v|\hat{y}_{s}|^{2} + vs^{2H-1}|\hat{z}_{s}|^{2}  \bigg) ds \nonumber \\
    &   + 2C \mathbb{E}\int_{t_{n}}^T e^{\beta s} \bigg(\frac{1}{v}\big(1+\frac{1}{s^{2H-1}}\big)|\hat{Y}_{s}|^{2}
        + v|\hat{y}_{s+\delta(s)}|^{2} + vs^{2H-1}|\hat{z}_{s+\zeta(s)}|^{2}  \bigg) ds \nonumber \\
\leq&  \frac{4C}{v} \mathbb{E}\int_{t_{n}}^T e^{\beta s}(1+\frac{1}{s^{2H-1}})|\hat{Y}_{s}|^{2} ds
        + 2C v \mathbb{E}\int_{t_{n}}^T e^{\beta s}\big(|\hat{y}_{s}|^{2} + s^{2H-1}|\hat{z}_{s}|^{2}\big) ds \nonumber \\
    &   + 2C v \mathbb{E}\int_{t_{n}}^T e^{\beta s} \big(|\hat{y}_{s+\delta(s)}|^{2} + s^{2H-1}|\hat{z}_{s+\zeta(s)}|^{2}  \big) ds \nonumber \\
\leq&  \frac{4C}{v} \mathbb{E}\int_{t_{n}}^T e^{\beta s}(1+\frac{1}{s^{2H-1}})|\hat{Y}_{s}|^{2} ds
        + 2C v(1+L) \mathbb{E}\int_{t_{n}}^{T+K} e^{\beta s}\big(|\hat{y}_{s}|^{2} + s^{2H-1}|\hat{z}_{s}|^{2}\big) ds, \nonumber
\end{align}
where $v>0$. Using the inequalities (\ref{304}) and (\ref{305}), and note that $M>2$, we obtain
\begin{align*}
  \mathbb{E}\left( \int_{t_{n}}^T e^{\beta s}|\hat{Y}_{s}|^{2} ds + \int_{t_{n}}^T e^{\beta s}s^{2H-1}|\hat{Z}_{s}|^{2} ds\right)
\leq \widetilde{G} \mathbb{E}\int_{t_{n}}^{T+K} e^{\beta s}\big(|\hat{y}_{s}|^{2} + s^{2H-1}|\hat{z}_{s}|^{2}\big) ds,
\end{align*}
or
\begin{align*}
  \mathbb{E} \int_{t_{n}}^{T+K} e^{\beta s}\big(|\hat{Y}_{s}|^{2} + s^{2H-1}|\hat{Z}_{s}|^{2}\big) ds
\leq \widetilde{G} \mathbb{E}\int_{t_{n}}^{T+K} e^{\beta s}\big(|\hat{y}_{s}|^{2} + s^{2H-1}|\hat{z}_{s}|^{2}\big) ds,
\end{align*}
where $\widetilde{G}=\frac{2CGM}{v}(T-t_{n}) + \frac{CGM}{v(1-H)}(T^{2-2H}-t_{n}^{2-2H}) + CM(1+L)v$.
Choosing $v$ such that $CM(1+L)v<\frac{1}{4}$,
and taking $n$ large enough such that
$$\frac{2CGM}{v}(T-t_{n})<\frac{1}{4}, \ \ \ \ \frac{CGM}{v(1-H)}(T^{2-2H}-t_{n}^{2-2H})<\frac{1}{4},$$
then
\begin{align*}
  \mathbb{E} \int_{t_{n}}^{T+K} e^{\beta s}\big(|\hat{Y}_{s}|^{2} + s^{2H-1}e^{\beta s}|\hat{Z}_{s}|^{2}\big) ds
\leq \frac{3}{4} \mathbb{E}\int_{t_{n}}^{T+K} e^{\beta s}\big(|\hat{y}_{s}|^{2} + s^{2H-1}|\hat{z}_{s}|^{2}\big) ds.
\end{align*}
Hence $I$ is a contraction on $\widetilde{\mathcal{V}}_{[t_{n},T+K]} \times \widetilde{\mathcal{V}}^{H}_{[t_{n},T+K]}$.
Arguing as in the proof of Theorem 22 in Maticiuc-Nie \cite{Maticiuc} we obtain that (\ref{31}) has a unique solution on $[t_{n},T]$.
The next step is to solve (\ref{31}) on  $[t_{n-1},t_{n}].$
In order to do this, one can show $I$ is a contraction on $\widetilde{\mathcal{V}}_{[t_{n-1},t_{n}+K]} \times \widetilde{\mathcal{V}}^{H}_{[t_{n-1},t_{n}+K]}$.
With the same arguments, repeating the above technique we obtain the anticipating BSDE (\ref{31}) admits a unique solution in
 $\widetilde{\mathcal{V}}_{[0,T+K]} \times \widetilde{\mathcal{V}}_{[0,T+K]}^{H}$.

Now we prove the estimate (\ref{20}).
Again from the It\^{o} formula, similarly as (\ref{36}),
\begin{align}
    &   \mathbb{E}\left(e^{\beta t}|Y_{t}|^{2} + \beta\int_t^T e^{\beta s}|Y_{s}|^{2} ds
      + \frac{2}{M}\int_t^T e^{\beta s}s^{2H-1}|Z_{s}|^{2} ds\right) \nonumber\\
\leq&   \mathbb{E}\left(e^{\beta T}|g(\eta_{T})|^{2}
      + 2\int_t^T e^{\beta s}Y_{s}f(s,\eta_{s},Y_{s},Z_{s},Y_{s+\delta(s)},Z_{s+\zeta(s)}) ds \right). \label{306}
\end{align}
By Lipschitz continuity of $f$,
\begin{align}
      & 2\mathbb{E}\int_t^T e^{\beta s} Y_{s} f(s,\eta_{s},Y_{s},Z_{s},Y_{s+\delta(s)},Z_{s+\zeta(s)}) ds \nonumber \\
  \leq& 2\mathbb{E}\int_t^T e^{\beta s}|Y_{s}| \bigg[C\big(|Y_{s}|+|Z_{s}|+|Y_{s+\delta(s)}|+|Z_{s+\zeta(s)}|\big)+|f_{0}(s,\eta_{s})|\bigg] ds \nonumber \\
  \leq&\mathbb{E}\int_t^T \bigg(2C+C^{2}+\frac{2C^{2}M}{s^{2H-1}}+\frac{2C^{2}ML}{s^{2H-1}}+1\bigg)e^{\beta s}|Y_{s}|^{2}ds
       + \frac{1}{2M}\mathbb{E}\int_t^T e^{\beta s}s^{2H-1}|Z_{s}|^{2}ds  \nonumber \\
      &+ \mathbb{E}\int_t^T e^{\beta s}|Y_{s+\delta(s)}|^{2}ds
       + \frac{1}{2ML}\mathbb{E}\int_t^T e^{\beta s}s^{2H-1}|Z_{s+\zeta(s)}|^{2}ds + \mathbb{E}\int_t^T e^{\beta s}|f_{0}(s,\eta_{s})|^{2}ds \nonumber \\
  \leq&\mathbb{E}\int_t^T \bigg(2C+C^{2}+\frac{2C^{2}M}{s^{2H-1}}+\frac{2C^{2}ML}{s^{2H-1}}+1+L\bigg)e^{\beta s}|Y_{s}|^{2}ds
           + \frac{1}{M}\mathbb{E}\int_t^T e^{\beta s}s^{2H-1}|Z_{s}|^{2}ds \nonumber \\
      &+L\mathbb{E}\int_T^{T+K} e^{\beta s}|g(\eta_{s})|^{2}ds + \frac{1}{2M}\mathbb{E}\int_T^{T+K} e^{\beta s}s^{2H-1}|h(\eta_{s})|^{2} dt
       + \mathbb{E}\int_t^T e^{\beta s}|f_{0}(s,\eta_{s})|^{2}ds. \label{307}
\end{align}
Combining (\ref{306}) and (\ref{307}),
\begin{align}
    &   \mathbb{E}\left(e^{\beta t}|Y_{t}|^{2} + \frac{1}{M}\int_t^T e^{\beta s}s^{2H-1}|Z_{s}|^{2} ds\right) \nonumber\\
\leq&  R\Theta(t,T,K) + \mathbb{E}\int_t^T \bigg(2C+C^{2}+L+1+\frac{2C^{2}M(L+1)}{s^{2H-1}}\bigg)e^{\beta s}|Y_{s}|^{2}ds. \label{308}
\end{align}
By Gronwall's inequality,
\begin{equation*}
   e^{\beta t}\mathbb{E}|Y_{t}|^{2} \leq R\Theta(t,T,K)\exp \bigg\{ (2C+C^{2}+L+1)(T-t) + 2C^{2}M(L+1)\frac{T^{2-2H}-t^{2-2H}}{2-2H} \bigg\}.
\end{equation*}
And by (\ref{308}) one also has
\begin{align*}
   \mathbb{E}\int_t^T e^{\beta s}s^{2H-1}|Z_{s}|^{2} ds \leq R\Theta(t,T,K).
\end{align*}
Hence the estimate (\ref{20}) is obtained.
This completes the proof.
\end{proof}

\section{Comparison Theorem}

In this section we study a comparison theorem for the anticipative BSDEs of the following form: for $i=1,2$,
\begin{equation}\label{41}
  \begin{cases}
    -dY_{t}^{i}=f^{i}(t,\eta_{t},Y^{i}_{t},Z^{i}_{t},Y^{i}_{t+\delta(t)}) dt - Z^{i}_{t} dB_{t}^{H}, \ \ \ t\in [0,T]; \\
      Y_{t}^{i}=g^{i}(\eta_{t}),  \ \ \ \ \ \ \ \ \ \ \ \ \ \ \ \ \ \ \ \ \ \ \ \ \ \ \ \ \ \ \ \ \ \ \ \ \ \ t\in[T,T+K].
  \end{cases}
\end{equation}
\begin{itemize}
\item[(H3)] There exists a constant $C\geq 0$, such that for all
$t\in [0,T], x,y,y',z \in \mathbb{R},$ $\theta_{\cdot},\theta'_{\cdot},\zeta_{\cdot} \in L_{\mathcal{F}}^2(t,T+K;\mathbb{R})$, $r\in[t,T+K]$, we have
\begin{equation*}
      |f(t,x,y,z,\theta_{r}) - f(t,x,y',z',\theta'_{r})|^{2}
 \leq C\big(|y-y'|^{2} + t^{2H-1}|z-z'|^{2} + \mathbb{E}^{\mathcal{F}_{t}}[|\xi_{r}-\xi'_{r}|^{2}]\big).
\end{equation*}
\end{itemize}

\begin{remark}
It's easy to see that (H3) is stronger then the Lipschitz condition in (H2),
so under (H1) and (H3), Eq. (\ref{41}) admits a unique solution
 $(Y^{i}_{\cdot},Z^{i}_{\cdot}) \in \widetilde{\mathcal{V}}_{[0,T+K]} \times \widetilde{\mathcal{V}}^{H}_{[0,T]}$.
We use (H3) to replace (H2) here for the reason that (H3) is more convenient for the proof of the following result.
\end{remark}

\begin{theorem}\label{40}
Let $\delta$ and $\zeta$ satisfy (i)-(ii). For $i=1,2$, suppose $g^{i}$ satisfies (H1), and $f^{i}$ and $f^{i}_{y}$ satisfy (H3).
Let $\overline{f}=\overline{f}(t,x,y,z,\theta_{\cdot})$ such that $\overline{f}$ and  $\overline{f}_{y}$  satisfy (H3), and for all
$(t,x,y,z) \in[0,T] \times \mathbb{R}^{3}$,
$\overline{f}(t,x,y,z,\cdot)$ is increasing, i.e.,
$\overline{f}(t,x,y,z,\theta_{r})\leq \overline{f}(t,x,y,z,\theta'_{r})$,
if $\theta_{r}\leq \theta'_{r}$, $\theta_{r},\theta_{r}'\in L^{2}_{\mathcal{F}}(t,T+K;\mathbb{R})$, $r\in [t,T+K]$.
Moreover, for all $(t,x,y,z) \in[0,T] \times \mathbb{R}^{3}$, $ \theta_{r}\in  L^{2}_{\mathcal{F}}(t,T+K;\mathbb{R})$,
\begin{equation*}
 f_{1}(t,x,y,z,\theta_{r})\leq \overline{f}(t,x,y,z,\theta_{r}) \leq f_{2}(t,x,y,z,\theta_{r}).
\end{equation*}
Then, if $g^{1}(x)\leq g^{2}(x)$,  $x\in \mathbb{R}$, we have
$$Y_{t}^{1}\leq Y_{t}^{2}, \ a.e., \ a.s.$$
\end{theorem}

\begin{proof}
Let $\overline{g}$ satisfies (H1) and
\begin{equation*}
  g^{1}(x)\leq \overline{g}(x)\leq g^{2}(x),  \ x\in \mathbb{R}.
\end{equation*}
Let $(\overline{Y}_{\cdot},\overline{Z}_{\cdot}) \in \widetilde{\mathcal{V}}_{[0,T+K]} \times \widetilde{\mathcal{V}}^{H}_{[0,T]}$ be the unique solution
of the following anticipative BSDE:
\begin{equation}\label{42}
  \begin{cases}
      \overline{Y}_{t}=\overline{g}(\eta_{T})+\int_t^T \overline{f}(s,\eta_{s},\overline{Y}_{s},\overline{Z}_{s},\overline{Y}_{s+\delta(s)}) ds
       - \int_t^T \overline{Z}_{s} dB_{s}^{H}, \  \ t\in [0,T]; \\
      \overline{Y}_{t}=\overline{g}(\eta_{t}),
       \ \ \ \ \ \ \ \ \ \ \ \ \ \ \ \ \ \ \ \ \ \ \ \ \ \ \ \ \ \ \ \ \ \ \ \ \ \ \ \ \ \ \ \ \ \ \ \ \ \ \ \ \ \ \ t\in[T,T+K].
  \end{cases}
\end{equation}
First we compare $\overline{Y}_{t}$ and $Y_{t}^{2}$.
Set $\widetilde{Y}_{0}(\cdot)=Y^{2}(\cdot)$ and consider the following BSDE:
\begin{equation*}
  \begin{cases}
      \widetilde{Y}_{1}(t)=\overline{g}(\eta_{T})
      +\int_t^T \overline{f}(s,\eta_{s},\widetilde{Y}_{1}(s),\widetilde{Z}_{1}(s),\widetilde{Y}_{0}(s+\delta(s))) ds
       - \int_t^T \widetilde{Z}_{1}(s) dB_{s}^{H}, \  t\in [0,T]; \\
      \widetilde{Y}_{1}(t)=\overline{g}(\eta_{t}),
        \ \ \ \ \ \ \ \ \ \ \ \ \ \ \ \ \ \ \ \ \ \ \ \ \ \ \ \ \ \ \ \ \ \ \ \ \ \ \ \ \ \ \ \ \ \ \ \ \ \ \ \ \ \ \ \ \ \ \ \ \ \ \ \ \ \ \ \ \ \ \
        t\in[T,T+K].
  \end{cases}
\end{equation*}
We see the above equation has a unique solution and denote it by $(\widetilde{Y}_{1}(\cdot),$ $\widetilde{Z}_{1}(\cdot))$.
Due to
\begin{equation*}
  \begin{cases}
   \overline{f}(s,x,y,z,\widetilde{Y}_{0}(s+\delta(s)))
   \leq f_{2}(s,x,y,z,\widetilde{Y}_{0}(s+\delta(s))), \ \ (s,x,y,z)\in [0,T]\times \mathbb{R}^{3};\\
   \overline{g}(x)\leq g^{2}(x), \ \ x\in \mathbb{R},
  \end{cases}
\end{equation*}
from Theorem 4.1 in Hu-Ocone-Song \cite{Hu1}, it follows that
\begin{equation*}
  \widetilde{Y}_{1}(t)\leq \widetilde{Y}_{0}(t)=Y^{2}(t), \  \ a.e., \ a.s.
\end{equation*}
Next, we consider the following BSDE:
\begin{equation*}
  \begin{cases}
      \widetilde{Y}_{2}(t)=\overline{g}(\eta_{T})
      +\int_t^T \overline{f}(s,\eta_{s},\widetilde{Y}_{2}(s),\widetilde{Z}_{2}(s),\widetilde{Y}_{1}(s+\delta(s))) ds
       - \int_t^T \widetilde{Z}_{2}(s) dB_{s}^{H}, \  t\in [0,T]; \\
      \widetilde{Y}_{2}(t)=\overline{g}(\eta_{t}),
     \ \ \ \ \ \ \ \ \ \ \ \ \ \ \ \ \ \ \ \ \ \ \ \ \ \ \ \ \ \ \ \ \ \ \ \ \ \ \ \ \ \ \ \ \ \ \ \ \ \ \ \ \ \ \ \ \ \ \ \ \ \ \ \ \ \ \ \ \ \ \
        t\in[T,T+K],
  \end{cases}
\end{equation*}
and let $(\widetilde{Y}_{2}(\cdot),\widetilde{Z}_{2}(\cdot)) \in \widetilde{\mathcal{V}}_{[0,T+K]} \times \widetilde{\mathcal{V}}^{H}_{[0,T]}$
be the unique solution of the above equation.
Now, since $\overline{f}(s,x,y,z,\cdot)$ is increasing,
 then for all $(s,x,y,z)\in [0,T]\times \mathbb{R}^{3}$,
\begin{equation*}
   \overline{f}(s,x,y,z,\widetilde{Y}_{1}(s+\delta(s)))
   \leq \overline{f}(s,x,y,z,\widetilde{Y}_{0}(s+\delta(s))).
\end{equation*}
Hence, similar to the above,
\begin{equation*}
  \widetilde{Y}_{2}(t)\leq \widetilde{Y}_{1}(t), \ \ a.e., \ a.s.
\end{equation*}
By induction, we can construct a sequence
$\{(\widetilde{Y}_{n}(\cdot),\widetilde{Z}_{n}(\cdot))\}_{n\geq 1} \subseteq \widetilde{\mathcal{V}}_{[0,T+K]} \times \widetilde{\mathcal{V}}^{H}_{[0,T]}$
such that
\begin{equation*}
  \begin{cases}
      \widetilde{Y}_{n}(t)=\overline{g}(\eta_{T})
      +\int_t^T \overline{f}(s,\eta_{s},\widetilde{Y}_{n}(s),\widetilde{Z}_{n}(s),\widetilde{Y}_{n-1}(s+\delta(s))) ds
       - \int_t^T \widetilde{Z}_{n}(s) dB_{s}^{H}, \ \ t\in [0,T]; \\
      \widetilde{Y}_{n}(t)=\overline{g}(\eta_{t}),     \ \ \ \ \ \ \ \ \ \ \ \ \ \ \ \ \ \ \ \ \ \ \ \ \ \ \ \ \ \ \ \ \ \ \ \ \ \ \ \ \ \ \ \ \ \ \ \ \ \ \ \ \ \ \ \ \ \ \ \ \ \ \ \ \ \ \ \ \ \ \ \ \ \ \ \  t\in[T,T+K].
  \end{cases}
\end{equation*}
Similarly, we obtain
\begin{equation*}
  Y^{2}(t)= \widetilde{Y}_{0}(t)\geq \widetilde{Y}_{1}(t)\geq \widetilde{Y}_{2}(t)\geq
   \cdots \geq \widetilde{Y}_{n}(t)\geq \cdots, \ a.e., \ a.s.
\end{equation*}
Next, we will show $\{(\widetilde{Y}_{n}(\cdot),\widetilde{Z}_{n}(\cdot))\}_{n\geq 1}$ is a Cauchy sequence. Set
$\hat{Y}_{n}(t)=\widetilde{Y}_{n}(t)-\widetilde{Y}_{n-1}(t),$ and $\hat{Z}_{n}(t)=\widetilde{Z}_{n}(t)-\widetilde{Z}_{n-1}(t), n\geq 4$.
Then by the estimate (\ref{34}) and (H3),
\[\begin{split}
    & \mathbb{E}\left(\frac{\beta}{2}\int_0^T e^{\beta s}|\hat{Y}_{n}(s)|^{2} ds +\frac{2}{M}\int_0^T s^{2H-1}e^{\beta s}|\hat{Z}_{n}(s)|^{2} ds\right)\\
\leq& \frac{2}{\beta}\mathbb{E}\bigg(\int_0^T e^{\beta s}
      | \overline{f}(s,\eta_{s},\widetilde{Y}_{n}(s),\widetilde{Z}_{n}(s),\widetilde{Y}_{n-1}(s+\delta(s)))\\
    & \ \ \ \ \ \ \ \ \ \ \ \ \ \ -\overline{f}(s,\eta_{s},\widetilde{Y}_{n-1}(s),\widetilde{Z}_{n-1}(s),\widetilde{Y}_{n-2}(s+\delta(s))) |^{2} ds\bigg)\\
\leq& \frac{2C}{\beta}\mathbb{E} \int_0^T e^{\beta s}\big(|\hat{Y}_{n}(s)|^{2}+s^{2H-1}|\hat{Z}_{n}(s)|^{2}\big) ds
     +\frac{2CL}{\beta}\mathbb{E} \int_0^T e^{\beta s}|\hat{Y}_{n-1}(s)|^{2} ds \\
\leq& \frac{2C(L+1)}{\beta}\mathbb{E} \int_0^T e^{\beta s}\big(|\hat{Y}_{n}(s)|^{2}+s^{2H-1}|\hat{Z}_{n}(s)|^{2}+|\hat{Y}_{n-1}(s)|^{2}\big) ds.
\end{split}\]
Let $\beta=8CM(L+1) + \frac{4}{M}$, then
\begin{equation*}
     \mathbb{E}\int_0^T e^{\beta s}\big(|\hat{Y}_{n}(s)|^{2} + s^{2H-1}|\hat{Z}_{n}(s)|^{2}\big) ds
\leq \frac{1}{4}\mathbb{E} \int_0^T e^{\beta s}\big(|\hat{Y}_{n}(s)|^{2}+s^{2H-1}|\hat{Z}_{n}(s)|^{2}+|\hat{Y}_{n-1}(s)|^{2}\big) ds.
\end{equation*}
Hence
\begin{equation*}
 \begin{split}
     \mathbb{E}\int_0^T e^{\beta s}\big(|\hat{Y}_{n}(s)|^{2} + s^{2H-1}|\hat{Z}_{n}(s)|^{2}\big) ds
\leq&  \frac{1}{3}\mathbb{E}\int_0^T e^{\beta s}|\hat{Y}_{n-1}(s)|^{2}  ds\\
\leq&  \frac{1}{3}\mathbb{E}\int_0^T e^{\beta s}\big(|\hat{Y}_{n-1}(s)|^{2} + s^{2H-1}|\hat{Z}_{n-1}(s)|^{2}\big) ds.
 \end{split}
\end{equation*}
So
\begin{equation*}
      \mathbb{E}\int_0^T e^{\beta s}(|\hat{Y}_{n}(s)|^{2} + s^{2H-1}|\hat{Z}_{n}(s)|^{2}) ds
\leq  (\frac{1}{3})^{n-4}\mathbb{E} \int_0^T e^{\beta s}(|\hat{Y}_{4}(s)|^{2} + s^{2H-1}|\hat{Z}_{4}(s)|^{2}) ds.
\end{equation*}
It follows that $(\hat{Y}_{n}(\cdot))_{n\geq 4}$ and $(\hat{Z}_{n}(\cdot))_{n\geq 4}$ are respectively Cauchy sequences in
$\widetilde{\mathcal{V}}_{[0,T+K]}$ and  $\widetilde{\mathcal{V}}_{[0,T]}^{H}$.
Denote their limits by $\widetilde{Y}_{\cdot}$ and $\widetilde{Z}_{\cdot}$, respectively.
%Since $\widetilde{\mathcal{V}}_{T+K}$ and $\widetilde{\mathcal{V}}^{H}_{T}$ are both Banach spaces, we obtain
%$(\widetilde{Y}_{\cdot},\widetilde{Z}_{\cdot})\in \widetilde{\mathcal{V}}_{T+K}\times \widetilde{\mathcal{V}}^{H}_{T}$, and
%$(\widetilde{Y}_{\cdot},\widetilde{Z}_{\cdot})$ satisfies the following anticipative BSDE:
%\begin{equation}\label{43}
%  \begin{cases}
%      \widetilde{Y}(t)=\overline{g}(\eta_{T})
%      +\int_t^T \overline{f}(s,\eta_{s},\widetilde{Y}(s),\widetilde{Z}(s),\widetilde{Y}(s+\delta(s))) ds
%       - \int_t^T \widetilde{Z}(s) dB_{s}^{H}, \  t\in [0,T]; \\
%      \widetilde{Y}(t)=\overline{g}(\eta_{t}),  \ \ \ \ \ \ \ \ \ \ \ \ \ \ \ \ \ \ \ \ \ \ \ \ \ \ \ \ \ \ \ \ \ \ \ \ \ \ \ \ \ \ \ \ \ \ \ \ \ \ \ \ \ \ \ \ \ \ \ \ \ \ \ \ \ \ \ t\in[T,T+K].
%  \end{cases}
%\end{equation}
%Comparing (\ref{42}) and (\ref{43}),
Now from Theorem \ref{30}, it follows that
\begin{equation*}
  \widetilde{Y}(t)= \overline{Y}(t), \ \ a.e., \ a.s.
\end{equation*}
Hence
\begin{equation*}
 \overline{Y}(t)\leq  Y^{2}(t), \ \ a.e., \ a.s.
\end{equation*}
Similarly, we can prove that
\begin{equation*}
Y^{1}(t)\leq \overline{Y}(t), \ \ a.e., \ a.s.
\end{equation*}
Therefore, the desired result is obtained.
\end{proof}

\section*{Acknowledgements}

We thank the referees for the detailed comments and suggestions to improve the paper significantly.

\end{document}